\RequirePackage{fix-cm}
\documentclass[smallextended]{svjour3}       
\smartqed  
\usepackage{graphicx}
\usepackage{amssymb,latexsym}
\usepackage{amscd}
\usepackage{euscript}
\usepackage{eufrak}
\usepackage{amsmath}

\begin{document}

\title{On Fidel--Vakarelov construction for Monadic G\"odel algebras
}


\author{Mar\'ia Valentina Alonso and Gustavo Pelaitay}


\institute{Mar\'ia Valentina Alonso and Gustavo Pelaitay  \at
              CONICET and Instituto de Ciencias B\'asicas \\
              Universidad Nacional de San Juan \\
              5400 San Juan Argentina \\
              \email{gpelaitay@gmail.com}
}


\maketitle

\begin{abstract} A significant correlation between Nelson algebras and Heyting algebras has been explored by several scholars, including Cignoli, Fidel, Vakarelov, and Sendlewski. This connection is integral to the concept of twist structures, whose origins can be traced back to the work of Kalman. In this paper, we obtain an expansion of the Fidel-Vakarelov construction, applying it to monadic Gödel algebras (or monadic prelinear Heyting algebras). This extension leads to the emergence of a new variety, which we aptly term monadic prelinear Nelson algebras.

\keywords{Heyting algebras; Nelson algebras, monadic Heyting algebras}
\subclass{06D20 \and 03G25}
\end{abstract}

\section{Introduction}
\label{intro}

Monadic Boolean algebras, as introduced by Halmos \cite{Halmos}, are Boolean algebras equipped with a closure operator denoted by $\exists$. This operator maps elements to a subalgebra within the Boolean algebra, abstracting the algebraic properties of the standard existential quantifier \emph{for some}. The term \emph{monadic} arises from its association with predicate logics used in languages featuring unary predicates and a single quantifier.

Extensive studies on monadic Boolean algebras have been conducted by Henkin and Nemeti \cite{Henkin,Nemeti}. Building on this foundation, the concept of monadic Heyting algebras was introduced as an algebraic model for the one-variable fragment of intuitionistic predicate logic \cite{Bezhanisvili,K1,K2,AM1,AM2,AM3}. Furthermore, monadic MV-algebras, which serve as the algebraic counterpart of monadic Łukasiewicz logic, have been introduced and thoroughly investigated \cite{DiNola19,DiNola04,R59}.

Subsequent developments led to the exploration of monadic basic algebras, monadic De Morgan algebras, monadic $LM_{n}^{m}$-algebras and monadic $k \times j$-rough Heyting algebras \cite{AP,Chajda08,Chajda09,Gallardo}.

Nelson algebras, also referred to as N-lattices and quasi-pseudo boolean algebras, were initially defined by Rasiowa \cite{R58}. They serve as the algebraic foundation for the intuitionistic propositional calculus featuring strong negation, as introduced by Nelson \cite{Nelson}. The connection between Nelson algebras and Heyting algebras has been extensively explored by various researchers, including Cignoli \cite{Cignoli}, Fidel \cite{Fidel1}, Vakarelov \cite{Vakarelov}, and Sendlewski \cite{Sendlewski}, among others. This association is a fundamental aspect of what is now recognized as twist structures \cite{Chajda22,O,FF,K98,R14}, with origins tracing back to \cite{Kalman}.

The definition of the functor from the category of Kleene algebras to the category of bounded distributive lattices given by Cignoli \cite{Cignoli} is based on Priestley
duality, and the interpolation property for Kleene algebras considered by Cignoli in establishing the equivalence is stated in topological terms. On the other hand, Sagastume proved in an unpublished manuscript \cite{Sagastume} that in centered Kleene
algebras the interpolation property is equivalent to an algebraic condition called (CK), that we will state later on. Moreover, she presented an equivalence between
the category of bounded distributive lattices and the category of centered Kleene algebras that satisfy (CK), but using a different (purely algebraic) construction to
that given by Cignoli in \cite{Cignoli}. In what follows we describe this equivalence whose details can be found in \cite{Castiglioni}.

Recall that a Kleene algebra is a De Morgan algebra denoted as $\langle T, \vee, \wedge, 0, 1 \rangle$ that satisfies the inequality $x \wedge \sim x \leq y \vee \sim y$. A Kleene algebra is termed centered if it possesses a center; that is, if there exists an element $c$ in $T$ such that $c = \sim c$. This element is necessarily unique.

 We write {\bf BDL} for the category of bounded distributive lattices and {\bf KAc} for the category
of centered Kleene algebras. In both cases the morphisms are the corresponding algebra homomorphisms. It is interesting to note that if $T$ and $U$ are centered Kleene algebras and $f:T\longrightarrow U$ is a morphism of Kleene algebras then $f$ preserves necessarily the center, i.e., $f(c)=c$.

The functor K from the category {\bf BDL} to the category {\bf KAc} is defined as follows.

For an object $A\in {\bf BDL}$ we let
$$K(A):=\{(a,b)\in A\times A: a\wedge b=0\}.$$
This set is endowed with the operations and the distinguished elements defined by:

\begin{align*} 
(a,b)\vee (d,e) &:=  (a\vee d,b\wedge e) \\ 
(a,b)\wedge (d,e) &:=  (a\wedge d,b\vee e) \\
\sim (a,b) &:=  (b,a) \\
0 &:=  (0,1) \\
 1 &:=  (1,0) \\
c &:=  (0,0)
\end{align*}

We have that $\langle K(A),\wedge,\vee,\sim,c,0,1\rangle\in {\bf KAc}.$ 

For a morphism $f:H\longrightarrow A\in {\bf BDL},$ the map $K(f):K(H)\longrightarrow K(A)$ defined by 

$$K(f)(a,b):=(f(a),f(b))$$

is a morphism in {\bf KAc}. Hence, K is a funtor from {\bf BDL} to {\bf KAc}.

Let $\langle T,\vee,\wedge,\sim,0,1\rangle\in {\bf KAc}$. The set 

$$C(T):=\{x\in T: x\geq c\}$$

is the universe of a subalgebra of $\langle T,\vee,\wedge, c,1\rangle$ and $\langle C(T),\vee,\wedge,c,1\rangle\in {\bf BDL}$. Moreover, if $g:T\longrightarrow U$ is a morphism in {\bf KAc}, then the map $C(g):C(T)\longrightarrow C(U),$ given by $C(g)(x)=g(x),$ is a morphism in {\bf BDL}. Thus, $C$ is a funtor from {\bf KAc} to {\bf BDL}.

Let $A\in {\bf BDL}$. The map $\alpha: A\longrightarrow C(K(A))$ given by $\alpha(a)=(a,0)$ is an isomorphism in {\bf BDL}. If $T\in {\bf KAc},$ then the map $\beta:T\longrightarrow K(C(T))$ given by $\beta(x)=(x\vee c,\sim x\vee c)$ is injective and a morphism in {\bf KAc}. It is not difficult to show that the functor $K:{\bf BDL}\longrightarrow {\bf KAc}$ has as left adjoint the functor $C:{\bf KAc}\longrightarrow {\bf BDL}$ with unit $\beta$ and counit $\alpha^{-1}$.

We are interested though in an equivalence between {\bf BDL} and the full subcategory of {\bf  KAc} whose objects satisfy the condition (CK) we proceed to state.

Let $T\in {\bf KAc}$. We consider the algebraic condition:

\begin{itemize}
\item [(CK)]$(\forall x,y\geq c)$ $(x\wedge y=c\Longrightarrow (\exists z)(z\vee c=x\, \&\, \sim z\vee c=y)).$
\end{itemize}

This condition characterizes the surjectivity of $\beta$ as demonstrated in \cite{Sagastume}. The condition (CK) is not necessarily verified in every centered Kleene algebra (see \cite{M63}).

\vspace{2mm}

We write ${\bf KA}^{\bf CK}_{c}$ for the full subcategory of {\bf KAc} whose objects satisfy {(CK)}. The functor K can then be seen as a functor from {\bf BDL} to ${\bf KA}^{\bf CK}_{c}$. The next theorem was
proved by Sagastume in \cite{Sagastume}.

\begin{theorem} The functors K and C establish a categorical equivalence between {\bf BDL}
and ${\bf KA}^{\bf CK}_{c}$ with natural isomorphisms $\alpha$ and $\beta$.
\end{theorem}

The previously described categorical equivalence can be restricted to the case of Heyting algebras and centered Nelson algebras. As we explain in what follows, this restriction allows for a more focused and insightful analysis of the relationship between these two types of algebras.

Recall that Nelson algebras (refer to \cite{M}) are algebraic structures ${\bf T}=\langle T,\vee,\wedge,\to,\sim,0,1\rangle$ that satisfy the conditions:

\begin{enumerate}
    \item [(N1)] $\langle T,\vee,\wedge,\sim,0,1\rangle$ is a Kleene algebra (see \cite{Balbes}),
    \item [(N2)] $x\to x=1,$
    \item [(N3)] $x\to (x\to z)=(x\wedge y)\to z,$
    \item [(N4)] $x\wedge (x\to y)=x\wedge (\sim x\vee y).$
\end{enumerate}

We define an algebra ${\bf T}=\langle T,\vee,\wedge,\to,\sim,0,1 \rangle$ as a \emph{centered Nelson algebra} if the reduct $\langle T,\vee,\wedge,\sim,0,1\rangle$ forms a centered Kleene algebra. Additionally, prelineal Nelson algebras, as described in \cite{Monteiro78}, constitute a subvariety of Nelson algebras characterized by the prelinearity equation $(x\to y)\vee (y\to x)=1$.

Furthermore, in \cite{M63}, Monteiro demonstrated that if $\langle  T,\vee,\wedge,\to,\sim,0,1\rangle$ constitute a Nelson algebra, then the following property is verified:

\begin{equation}\label{RN}
x\wedge z\leq \sim x\vee y \Longleftrightarrow z\leq x\to y.    
\end{equation}

We denote by {\bf HA} the category of Heyting algebras and by {\bf NA}$_{c}$ for the category of centered Nelson algebras.

Fidel \cite{Fidel1} and  Vakarelov \cite{Vakarelov} proved independently that if $A \in {\bf HA}$, then the Kleene algebra $K(A)$ is a centered Nelson algebra, in which the weak implication is defined for pairs $(a, b)$
and $(d, e)$ in $K(A)$ as follows:

\begin{equation}\label{implication}
(a,b)\to (d,e):=(a\Rightarrow d,  a\wedge e).
\end{equation}

The following result appears in \cite[Theorem 3.14]{Cignoli}.

\begin{theorem} The functors K and C establish a categorical equivalence between {\bf HA}
and {\bf NA}$_c$ with natural isomorphisms $\alpha$ and $\beta$.
\end{theorem}

\section{Preliminaries}

In this section, we summarize some definitions and results about monadic G\"odel, which will be used in the following sections.

\begin{definition} Heyting algebras are algebras $\langle A,\vee,\wedge,\Rightarrow,0,1\rangle$ that satisfy the conditions:

\begin{itemize}
\item [(h1)] $\langle A,\vee,\wedge,0,1\rangle$ is a bounded lattice.
\item [(h2)] $x\wedge (x\Rightarrow y)=x\wedge y$.
\item [(h3)] $x\wedge (y\Rightarrow z)=x\wedge [(x\wedge y)\Rightarrow (x\wedge z)].$
\item [(h4)] $(x\wedge y)\Rightarrow x=1$.
\end{itemize}

In what follows, we will denote the Heyting algebra $\langle A,\vee,\wedge,\Rightarrow,0,1\rangle$ as ${\bf A}.$

\end{definition}

\begin{definition}{\rm (\cite{Bezhanisvili,AM2})}
A monadic Heyting algebra is a structure $({\bf A}, \forall, \exists)$, where ${\bf A}$ is a Heyting algebra and $\forall$ and $\exists$ are unary operations verifying the following identities:

\begin{flushleft} 
 \begin{minipage}[t]{0.5\textwidth}
 \begin{itemize}
 \item [(m1)] $\forall x\leq x,$
     \item[(m2)] $\forall(x\wedge y)=\forall x\wedge\forall y,$
     \item[(m3)] $\forall1=1,$
     \item[(m4)] $\forall\exists x=\exists x,$
     \item[(m5)] $\forall (x\Rightarrow y)\leq \exists x\Rightarrow\exists y$.
\end{itemize}
 \end{minipage}\hfill\begin{minipage}
 [t]{0.5\textwidth}
 $x\leq \exists x$,\newline 
 $\exists(x\vee y)=\exists x\vee\exists y$,\newline
 $\exists 0=0$,\newline
 $\exists\forall x= \forall x$,
\end{minipage}
\end{flushleft}
\end{definition}

G\"odel algebras are prelinear Heyting algebras, that is, they constitute the variety generated
by totally ordered Heyting algebras. Concretely, G\"odel algebras are the subvariety of Heyting algebras determined by the prelinearity equation $(x \Rightarrow y) \vee (y \Rightarrow x) = 1$. More precisely, monadic G\"odel algebras coincide with monadic prelinear Heyting algebras that satisfy the equation 

\begin{equation}\label{G}
\forall (\exists x\vee y)=\exists x\vee \forall y.  
\end{equation}

We conclude this section by summarizing the fundamental properties of monadic Gödel algebras; all the proofs can be found in \cite{CC} within the broader context of monadic BL-algebras.

The next lemma collects some of the basic properties that hold true in any monadic G\"odel algebra.    

\begin{lemma} Let $({\bf A},\forall,\exists)$ be a monadic G\"odel algebra. Then,

\begin{enumerate}
    \item $\exists (A)=\forall (A);$
    \item $\exists (A)$ is a subalgebra of  {\bf A};
    \item $\exists a=\min \{b\in \exists(A): b\geq a\}$ and $\forall a=\max \{b\in \exists (A): b\leq a\}$ for every $a\in A;$
    \item the lattices of congruences of $({\bf A},\forall,\exists)$ and $\exists (A)$ are isomorphic;
    \item $({\bf A},\forall,\exists)$ is finitely subdirectly irreducible if and only if $\exists(A)$ is totally ordered;
    \item $({\bf A},\forall,\exists)$ is subdirectly irreducible if and only if $\exists(A)$ is totally ordered and there exists $u\in \exists(A)\setminus\{1\}$ such that $a\leq u$ for all $a\in \exists(A)\setminus \{1\}.$
\end{enumerate}

\end{lemma}

The next lemma includes several arithmetical properties, some of which are used constantly throughout the paper.

\begin{lemma}\label{lemma2} Let $({\bf A},\forall,\exists)$ be a monadic G\"odel algebra. Then, for any $a,b\in A$ and $u\in \exists (A):$

\begin{enumerate}
    \item $\exists 1=1$ and $\forall 0=0;$
    \item $\forall u=u=\exists u;$
    \item $\forall a\leq a\leq \exists a;$
    \item if $a\leq b,$ then $\forall a\leq \forall b$ and $\exists a\leq \exists b;$
    \item $\forall(a\vee u)=\forall a\vee u;$
    \item $\exists(a\wedge u)=\exists a\wedge u;$
    \item $\forall(a\Rightarrow u)=\exists a\Rightarrow u;$
    \item $\exists (a\Rightarrow u)\leq \forall a\Rightarrow u;$
    \item $\forall (u\Rightarrow a)=u\Rightarrow \forall a;$
    \item $\exists(u\Rightarrow a)\leq u\Rightarrow \exists a;$
    \item $\exists a\wedge \forall b\leq \exists (a\wedge b);$
    \item $\forall (a\Rightarrow b)\leq  \forall a\Rightarrow \forall b;$
    \item $\forall \neg a=\neg \exists a;$
    \item $\exists \neg a\leq \neg \forall a.$
\end{enumerate}

\end{lemma}

\section{Monadic Nelson algebras}

In this section, we will introduce the concept of monadic centered Nelson algebras and present some fundamental properties.

\begin{definition}\label{d3}
An algebra ${\bf U}=({\bf T},\exists)$ is a monadic Nelson algebra if ${\bf T}=\langle T,\vee,\wedge,\rightarrow,\sim,0,1\rangle$ is a Nelson algebra and the following conditions hold:
\begin{enumerate}
    \item[(n1)] $\exists 0=0,$
    \item[(n2)] $x\leq \exists x,$
    \item[(n3)] $\exists (x\wedge \exists y)=\exists x\wedge \exists y,$
    \item[(n4)] $\exists (x\vee y)=\exists x\vee \exists y,$
    \item[(n5)] $\forall \exists x=\exists x,$
    \item[(n6)]  $\forall (x\to y)\leq \exists x\to \exists y,$
    \item[(n7)] $\forall (x\to y)\leq \forall x\to \forall y$,  where $\forall a:=\sim \exists (\sim a).$
\end{enumerate}
Furthermore, if ${\bf T}$ is a prelinear Nelson algebra with center $c$, and $c \in \exists(T)$, we refer to ${\bf U}$ as a monadic centered prelinear Nelson algebra (or monadic $N_c$-algebra for brevity).\end{definition}

\begin{remark} Let ${\bf U}=({\bf T},\exists)$ be a monadic Nelson algebra. Then, from (n1) to (n5), we have that the reduct  $\langle T,\vee,\wedge,\sim,\exists,0,1\rangle$ is a monadic De Morgan algebra (see \cite{Chajda09}).
\end{remark}

The proof of the following algebraic properties are straightforward.

\begin{lemma}\label{propiedadesNelsonMonadicas} Let ${\bf U}=({\bf T},\exists)$ be a monadic Nelson algebra. Then, the following properties hold:

\begin{enumerate}
    \item $\forall 1=1,$
    \item $\forall x\leq x,$
    \item $\forall(x\vee \forall y)=\forall x\vee \forall y,$
    \item $\forall (x\wedge y)=\forall x\wedge \forall y,$
    \item $\exists \forall x=\forall x,$
    \item $\forall\exists x=\exists x.$
\end{enumerate}

\end{lemma}

\section{Fidel--Vakarelov construction}

In this section, we prove some results that establish the connection between monadic G\"odel algebras and monadic $N_c$-algebras.

Let $(\mathbf{A},\forall,\exists)$ be a monadic Heyting algebra and let us consider

$$K(A):=\{(a,b)\in A\times A: a\wedge b=0\}.$$

It is well known from \cite{Fidel1,Vakarelov} that by defining: 

\begin{align*} 
(a,b)\vee (d,e) &:=  (a\vee d,b\wedge e) \\ 
(a,b)\wedge (d,e) &:=  (a\wedge d,b\vee e) \\
(a,b)\to (d,e) &:=(a\Rightarrow d,  a\wedge e)\\
\sim (a,b) &:=  (b,a) \\
0 &:=  (0,1) \\
 1 &:=  (1,0) \\
c &:=  (0,0)
\end{align*}

we get that the structure $\mathbf{A}_{K}=\langle K(A),\vee,\wedge,\sim, \to,c,0,1\rangle$ is a  centered Nelson algebra.

Now, we define on $K(A)$ the following unary operators: 
\begin{eqnarray}\label{Existe}
  \exists_K(a,b)  =  (\exists a,\forall b), \,\,\, \forall_K(a,b)  =  (\forall a,\exists b).
\end{eqnarray}

\begin{lemma}\label{l3} Let ${\bf G}=({\bf A},\forall,\exists)$ be a monadic G\"odel algebra and let $(a,b)\in K(A)$. Then, the following hold:

\begin{itemize}
\item[{\rm (a)}] $\exists_{K}(a,b)\in K(A),$
\item [{\rm (b)}] $\forall_{K}(a,b)=\sim \exists_K(\sim (a,b)),$
\item [{\rm (c)}] $\forall_{K}(a,b)\in K(A).$
\end{itemize}

\end{lemma}

\begin{proof} We will only prove (a). Let $(a,b) \in K(A)$. Therefore, $a \wedge b = 0$. Then, by using properties (m3) and 11 from Lemma \ref{lemma2}, we have $\exists a \wedge \forall b \leq \exists (a \wedge b) = \exists 0 = 0$. Consequently, $\exists_{K}(a,b) \in K(A)$.
\end{proof}

\begin{lemma}\label{l4} Let ${\bf G}=({\bf A},\forall,\exists)$ be a monadic G\"odel algebra. Then, $${\bf K}({\bf G})=({\bf A}_{\bf K},\exists_{K})$$ is a monadic $N_c$-algebra.
\end{lemma}

\begin{proof} It is well known that ${\bf A}_{\bf K}$ is a centered Nelson algebra. Now, let's prove that $K(A)$ satisfies the axiom of prelinearity. Let $(a,b), (x,y)\in K(A)$. Then, $[(a,b) \to (x,y)] \vee [(x,y) \to (a,b)]=(a\Rightarrow x, a\wedge y)\vee (x\Rightarrow a, x\wedge b)=((a\Rightarrow x)\vee (x\Rightarrow a), a\wedge y\wedge x\wedge b)=(1,0)$. Referring to Lemma \ref{l3}, it becomes evident that the operator $\exists_K$ is well-defined. Now, we will proceed to prove the axioms {\rm (n1)} to {\rm (n6)} from Definition \ref{d3}.

\begin{itemize}
\item[{\rm (n1)}:] From (m3), we have that $\exists_{K}(0,1)=(\exists 0, \forall 1)=(0,1).$
\item [{\rm (n2)}:] From (m1), we obtain $(a,b)\wedge (\exists a, \forall b)=(a\wedge \exists a,b\vee \forall b)=(a,b)$. Therefore, $(a,b)\leq \exists_{K}(a,b)$.

\item [{\rm (n3)}:] From 5 and 6 of Lemma \ref{lemma2}, we have that $\exists_{K}[(a,b)\wedge \exists_K (x,y)]=\exists_{K}[(a,b)\wedge (\exists x,\forall y)]=\exists_{K}(a\wedge \exists x, b\vee \forall y)=(\exists (a\wedge \exists x),\forall(b\vee \forall y))=(\exists a\wedge \exists x, \forall b\vee \forall y)=(\exists a,\forall b)\wedge (\exists x, \forall y )=\exists_{K}(a,b)\wedge \exists_{K}(x,y).$

\item [{\rm (n4)}:] From (m2), we have that $\exists_{K}[(a,b)\vee (x,y)]=\exists_{K}(a\vee x,b\wedge y)=(\exists (a\vee x),\forall(b\wedge y))=(\exists a\vee \exists x,\forall b\wedge \forall y)=(\exists a,\forall b)\vee (\exists x,\forall y)=\exists_{K}(a,b)\vee \exists_{K}(x,y).$

\item [{\rm (n5)}:] From (m4), we deduce that $\forall_{K}\exists_{K}(a,b) = \forall_{K}(\exists a,\forall b) = (\forall\exists a,\exists\forall b) = (\exists a, \forall b) = \exists_{K}(a,b)$. Then, using (b) from Lemma \ref{l3}, we derive the desired identity.

\item[{\rm (n6)}] From (m5), we deduce that $\forall(a\Rightarrow x)\leq \exists a\Rightarrow \exists x$, and from item 11 in Lemma \ref{lemma2}, we conclude $\exists a\wedge \forall y\leq \exists (a\wedge y)$. Therefore, $(\forall(a\Rightarrow x),\exists(a\wedge y))\leq (\exists a\Rightarrow \exists x,\exists a\wedge \forall y)$. Thus, we can establish $\forall_{K}((a,b)\to (x,y))\leq \exists_{K}(a,b)\to \exists_{K}(x,y).$

\item[{\rm (n7)}] From Lemma \ref{lemma2}, we have that $\forall(a \Rightarrow x)\leq \forall a \Rightarrow \forall x$, and $\forall a \wedge \exists y\leq \exists(a\wedge y)$. This implies that $(\forall(a\Rightarrow x),\exists(a\wedge y))\leq (\forall a\Rightarrow\forall x,\forall a\wedge \exists y)$. Therefore, using \ref{G}, we can conclude that $\forall_K((a,b)\rightarrow(x,y))\leq \forall_K(a,b)\rightarrow\forall_ K (x,y)$.

\end{itemize}

\end{proof}

The following remark illustrates that equation \ref{G} cannot be omitted in Lemma \ref{l4}. This observation, in our view, justifies the study of the Fidel-Vakarelov construction in the case of monadic G\"odel algebras.

\begin{remark}
Note that monadic prelinear Heyting algebras may not satisfy equation \ref{G}. A
counterexample is given by the monadic Heyting algebra $(A,\exists,\forall)$ depicted in the Hasse diagram below
with the monadic operators defined as in the table.

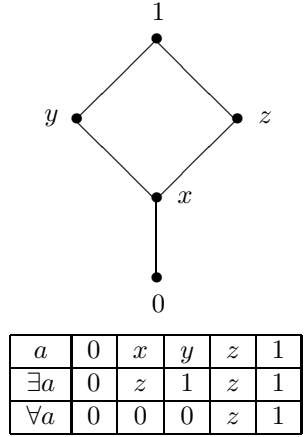
\begin{center}
\hspace{0.25cm}
\put(00,00){\makebox(1,1){$\bullet$}}
\put(00,30){\makebox(1,1){$\bullet$}}
\put(-30,60){\makebox(1,1){$\bullet$}}
\put(30,60){\makebox(1,1){$\bullet$}}
\put(00,90){\makebox(1,1){$\bullet$}}
\put(00,00){\line(0,0){30}}
\put(00,30){\line(1,1){30}}
\put(00,30){\line(-1,1){30}}
\put(-30,60){\line(1,1){30}}
\put(30,60){\line(-1,1){30}}
\put(00,-10){\makebox(2,2){$ 0$}}
\put(10,30){\makebox(2,2){$ x$}}
\put(-40,60){\makebox(2,2){$ y$}}
\put(40,60){\makebox(2,2){$ z$}}
\put(00,100){\makebox(2,2){$ 1$}}
\end{center}
\begin{center}
\begin{tabular}{|c|c|c|c|c|c|}\hline
 $a$   & $0$ & $x$ & $y$ & $z$ & $1$ \\ \hline
$\exists a$ & $0$ & $z$ & $1$ & $z$ & $1$  \\ \hline
$\forall a$ & $0$ & $0$ & $0$ & $z$ & $1$  \\ \hline
\end{tabular}
\end{center}

Indeed, note that $\forall(y \vee \exists z) = \forall(y \vee z) = \forall 1 = 1$ whereas $\forall z \vee \exists z = 0 \vee z = z$.

\

Upon applying the previously described Fidel-Vakarelov construction, we arrive at the ensuing centered Nelson algebra: $$K(A)=\{ (0,0), (0,1), (1,0), (0,x), (x,0), (0,y), (y,0), (0,z), (z,0) \}.$$ This is elucidated in the subsequent Hasse diagram.

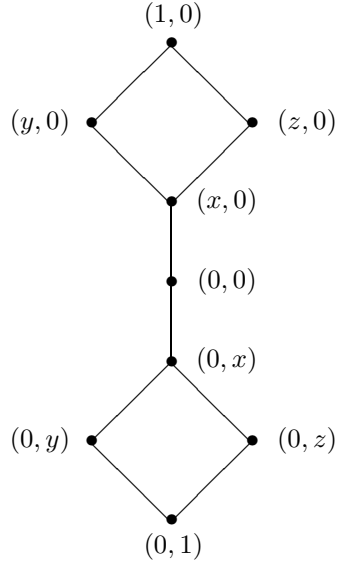
\begin{center}
\hspace{0.25cm}
\put(00,00){\makebox(1,1){$\bullet$}}
\put(-30,30){\makebox(1,1){$\bullet$}}
\put(30,30){\makebox(1,1){$\bullet$}}
\put(00,60){\makebox(1,1){$\bullet$}}
\put(00,90){\makebox(1,1){$\bullet$}}
\put(00,120){\makebox(1,1){$\bullet$}}
\put(-30,150){\makebox(1,1){$\bullet$}}
\put(30,150){\makebox(1,1){$\bullet$}}
\put(00,180){\makebox(1,1){$\bullet$}}
\put(00,00){\line(1,1){30}}
\put(00,00){\line(-1,1){30}}
\put(-30,30){\line(1,1){30}}
\put(30,30){\line(-1,1){30}}
\put(00,60){\line(0,0){30}}
\put(00,90){\line(0,0){30}}
\put(00,120){\line(1,1){30}}
\put(00,120){\line(-1,1){30}}
\put(-30,150){\line(1,1){30}}
\put(30,150){\line(-1,1){30}}
\put(00,-10){\makebox(2,2){$ (0,1)$}}
\put(-50,30){\makebox(2,2){$ (0,y)$}}
\put(50,30){\makebox(2,2){$ (0,z)$}}
\put(20,60){\makebox(2,2){$ (0,x)$}}
\put(20,90){\makebox(2,2){$(0,0)$}}
\put(20,120){\makebox(2,2){$ (x,0)$}}
\put(-50,150){\makebox(2,2){$ (y,0)$}}
\put(50,150){\makebox(2,2){$ (z,0)$}}
\put(00,190){\makebox(2,2){$ (1,0)$}}
\end{center}

Let's observe that if we consider the operator $\exists_{K}$ defined as in Equation \ref{Existe}, axiom (n3) from the Definition \ref{d3} is not satisfied. Indeed:
\[
\exists_K ((x,0)\wedge \exists_K (0,x))=(0,z)\neq (0,0)=\exists_K (x,0) \wedge\exists_K (0,x).
\]

Hence, $(K(A), \exists_K)$ is not a monadic De Morgan algebra, and consequently, it is not a monadic $N_{c}$-algebra.

\end{remark}

We write $\textbf{mG}$ for the category whose objects are monadic G\"odel algebras and \textbf{mN}$_{c}$ for the
category whose objects are monadic $N_{c}$-algebras. In both cases, the morphisms are the corresponding algebra homomorphisms. Moreover, if ${\bf G}=({\bf A},\forall,\exists)$ and ${\bf M}=(B,\forall,\exists)$ are monadic G\"odel algebras and $f:{\bf G}\longrightarrow {\bf M}$ is a morphism in \textbf{mG}, then it is no hard to see that the map $K(f):K(A)\longrightarrow K(B)$ given by $K(f)(x,y)=(f(x),f(y))$ is a morphism in {\bf mN}$_{c}$ from ${\bf K}({\bf G})=(\textbf{A}_\textbf{K},\exists_{K})$ to ${\bf K}({\bf M})=(\textbf{B}_\textbf{K},\exists_{K})$. It is clear that these assignments establish a functor ${\bf K}$ from {\bf mG} to {\bf mN}$_{c}.$

\begin{lemma}
Let ${\bf T}=\langle T,\wedge,\vee,\to,\sim,c,0,1\rangle$ be a centered prelineal Nelson algebra and define

$$C(T):=\{x\in T: x\geq c\}.$$

Then, the structure

$${\bf T}_{C}=\langle C(T),\wedge,\vee,\to,c,1\rangle$$ is a G\"odel algebra. Moreover, if $f:{\bf T}\longrightarrow {\bf S}$ is a homomorphism of centered Nelson algebras, then if follows that $C(f):C(T)\longrightarrow C(S),$ defined by $C(f)(x)=f(x),$ is a homomorphism of G\"odel algebras.

\end{lemma}

\begin{proof} It is evident that $C(T)$ constitutes a bounded distributive lattice. Now, let us proceed to establish that it is a G\"odel algebra. Take $x$ and $y$ from $C(T)$. Given $x \geq c$, we have $\sim c = c \leq \sim x$. Extending this, as $y \geq c$, it follows that $c \leq \sim x \vee y$. This implies $x \wedge c \leq \sim x \vee y$. Consequently, referring to \ref{RN}, we conclude that $c \leq x \rightarrow y.$  Let \(x, y, z \in C(T)\). Assuming \(x \wedge y \leq z\), we can deduce \(x \wedge y \leq \sim y \vee z\), which implies by \ref{RN}, \(x \leq y \rightarrow z\). Conversely, if \(x \leq y \rightarrow z\), then \(x \wedge y \leq (y \rightarrow z) \wedge y = y \wedge (\sim y \vee z) \leq \sim y \vee z\). Furthermore, since \(\sim y \leq \sim c = c\) and \(c \leq z\), we conclude that \(x \wedge y \leq z\). Taking into account the above, since the axiom of prelinearity is satisfied in $T$, we conclude that $C(T)$ is a G\"odel algebra. Finally, showing that $C(f)$ is a homomorphism of G\"odel algebras is straightforward.

\end{proof}

\

\begin{lemma}\label{l5} Let ${\bf U}=({\bf T},\exists)$ be a monadic $N_c$-algebra. Then, $C({\bf U})=({\bf T}_{C},\forall,\exists)$ is a monadic G\"odel algebra. Moreover, if $f:({\bf T},\exists)\longrightarrow ({\bf S},\exists)$ is a morphism in {\bf mN}$_{c}$, then $C(f)$ es a morphism in {\bf mG}.    
\end{lemma}

\begin{proof} We will only verify that $C(T)$ is closed under $\exists$ and $\forall$. The remaining part of the proof is left to the reader. Let $x \in C(T)$. Then, $x \geq c$. Consequently, since $\exists$ is a monotone operator, we obtain $\exists x \geq \exists c = c$. Thus, $\exists x \in C(T)$. On the other hand, let's provide a proof that $\forall x \in C(T)$. Given that $x \geq c$, we have $\sim x \leq c$. Again, considering that $\exists$ is monotone, it follows that $\exists \sim x \leq \exists c = c$. Hence, $c \leq \sim \exists \sim x $. Therefore, $\forall x \in C(T)$.
\end{proof}

Let $ f : (T, \exists) \to (S, \exists) $ be a homomorphism of monadic $N_c$-algebras. It is clear now, from Lemma \ref{l5}, that the assignments $ T \mapsto C(T) $, $ f \mapsto C(f) $ determine a functor $ C$  from $ {\bf mN}_{c}$ to $ \bf{mG}$.

\begin{lemma} Let ${\bf G}=({\bf A},\forall,\exists)$ be a monadic G\"odel algebra. Then the map $\alpha: {\bf G}\longrightarrow C({\bf K}({\bf G}))$ given by $\alpha(x)=(x,0)$ is an isomorphism in ${\bf mG}$. \end{lemma}

\begin{proof} We will only prove that $\alpha$ commutes with the unary operators $\exists$ and $\forall$.
Let $a \in A$. Then 
\begin{itemize}
\item $\alpha(\forall a) = (\forall a, 0) = (\forall a, \exists 0) = \forall_{K}((a, 0)) =\forall_{K}(\alpha(a))$.
\item $\alpha(\exists a) = (\exists a, 0) = (\exists a, \forall 0) = \exists_{K}((a, 0)) = \exists_{K}(\alpha(a))$.
\end{itemize}
    
\end{proof}

\begin{remark}
In \cite[Theorem 2.4]{Cignoli}, it was proved that there exists a categorical equivalence between the categories {\bf BDL} and the full subcategory of {\bf KA}$_{c}$, consisting of objects satisfying a topological condition known as the interpolation property. M. Sagastume \cite{Sagastume} subsequently noted that the interpolation property is equivalent to the algebraic condition (CK). Furthermore, Cignoli in \cite{Cignoli} established that every Nelson algebra $T$ satisfies the interpolation property. Thus, based on the aforementioned findings, it can be concluded that the (CK) property holds in all Nelson algebras. This implies that $\beta:T\longrightarrow K(C(T))$ defined by $\beta(x)=(x\vee c, \sim x\vee c)$ is an isomorphism of Nelson algebras.    
\end{remark}

Next, we establish that the previously mentioned result can be extended to the monadic context.

\begin{lemma} Let $U=(T,\exists)$ be a monadic $N_{c}$-algebra. Then the map $\beta$ is a isomorphism in {\bf mN}$_{c}$.    
\end{lemma}

\begin{proof} We will only prove that $\beta$ commutes with the unary operator $\exists$. Let $a \in A$. By using properties (m2) and 6 of Lemma \ref{lemma2}, we can conclude that
\[
\begin{array}{lll}
\beta(\exists a) & = & (\exists a \vee c, (\sim \exists a) \vee c) \\
                         & = & (\exists a \vee c, \forall (\sim  a) \vee c) \\
                         & = & (\exists (a \vee c), \forall (\sim a \vee c)), \\
                        
\end{array}
\]
i.e., $\beta(\exists a) = \exists_{K}(\beta(a))$.

\end{proof}

Straightforward computations based on previous results of this section prove the following result.

\begin{theorem} \label{equivalence theorem}
The functors K and C establish a categorical equivalence between ${\bf mG}$ and {\bf mN}$_{c}$ with natural isomorphisms $\alpha$ and $\beta$.
\end{theorem}

\section{Congruences}

Write $Con({\bf A})$ for the lattice of congruences of an algebra ${\bf A}$. Let {\bf L} be a bounded distributive lattice. If $\theta\in Con({\bf L}),$ we can define a congruence $\gamma_{\theta}$ of $K({\bf L})$ by 

\begin{center}
$(a,b)\gamma_{\theta}(x,y)$ if and only if $(a,x)\in \theta$ and $(b,y)\in \theta.$    
\end{center}

Reciprocally, if $\gamma\in Con(K({\bf L})),$ we can also define a congruence $\theta^{\gamma}$ of ${\bf L}$ as

\begin{center}
$(a,b)\in \theta^{\gamma}$ if and only if $(a,0)\gamma (b,0).$
\end{center}

In Lemma 5.3 of \cite{Castiglioni} it was proved that the assignments $\theta\mapsto \gamma_\theta$ and $\gamma\mapsto \theta^{\gamma}$ establish an order isomorphism between $Con({\bf L})$ and $Con(K({\bf L})).$

\

The following two results prove that the latter assignment can also be extended to monadic G\"odel algebras and monadic $N_{c}$-algebras.

\begin{lemma} Let ${\bf G}=({\bf A},\forall,\exists)$ be a monadic G\"odel algebra and let $\theta\in Con({\bf G})$ and $\gamma\in Con(K({\bf G}))$. Then, $\gamma_\theta\in Con(K({\bf G}))$ and $\theta^{\gamma}\in Con({\bf G}).$    
\end{lemma}

\begin{proof} Let $(a,b)\gamma_\theta (x,y)$ and $(v,w)\gamma_\theta(t,u),$ meaning that $(a,x)\in \theta,$ $(b,y)\in \theta,$ $(v,t)\in \theta,$ and $(w,u)\in \theta$. Consequently, we have $(a\Rightarrow v,x\Rightarrow t)\in \theta$ and $(a\wedge w,x\wedge u)\in \theta$. By the definition of $\gamma_\theta$, it follows that $(a\Rightarrow v,a\wedge w)\gamma_\theta(x\Rightarrow t, x\wedge u)$. Hence, $(a,b)\to (v,w)\gamma_\theta (x,y)\to (t,u).$

On the other hand, if $(a,b)\gamma_\theta (x,y)$, which implies that $(a,x)\in \theta$ and $(b,y)\in \theta$, then, because $\theta$ is compatible with $\exists$ and $\forall$, we can conclude that $(\exists a,\exists x)\in \theta$ and $(\forall b,\forall y)\in \theta$. This leads to $(\exists a,\forall b)\gamma_\theta (\exists x,\forall y)$. So, $\exists_{K}(a,b)\gamma_\theta \exists_{K}(x,y).$  Therefore, $\gamma_\theta\in Con(K({\bf G}))$.  

Finally, let's prove that $\theta^{\gamma}$ belongs to $Con({\bf G})$. Take $(a,b)\in \theta^{\gamma}$ and $(x,y)\in \theta^{\gamma}$. This implies $(a,0)\gamma (b,0)$ and $(x,0)\gamma (y,0)$. Consequently, we have $(a,0)\to (x,0)\gamma (b,0)\to (y,0)$. However, $(a,0)\to (x,0)=(a\Rightarrow x,0)$ and $(b,0)\to (y,0)=(b\Rightarrow y,0)$. Thus, we can conclude that $(a\Rightarrow x, b\Rightarrow y)\in \theta^{\gamma}$.

Now, let's demonstrate that $\theta^{\gamma}$ is compatible with both $\exists$ and $\forall$. Assume $(a,b)\in \theta^{\gamma}$. This means $(a,0)\gamma (b,0)$. Given that $\gamma$ is compatible with $\exists_K$, we can deduce $(\exists a,0)\gamma (\exists b,0)$. Hence, we have $(\exists a,\exists b)\in \theta^{\gamma}$. Additionally, since $\gamma$ is compatible with $\sim$ and $\exists_K$, we know that $\forall_{K}(a,0)\gamma \forall_{K}(b,0)$, which leads to $(\forall a,0)\gamma (\forall b,0)$. In other words, $(\forall a,\forall b)\in \theta^{\gamma}$.

\end{proof}

\begin{theorem} Let ${\bf G}=({\bf A},\forall,\exists)$ be a monadic G\"odel algebra. Then, the mapping $f:Con({\bf G})\longrightarrow Con(K({\bf G}))$, defined as $f(\theta)=\gamma_\theta$, establishes an order isomorphism.    
\end{theorem}

\section*{Conclusion and Opens Problems}

In this paper, we prove that the Fidel-Vakarelov construction cannot be extended to the context of monadic Heyting algebras and monadic centered Nelson algebras. However, for the case of monadic G\"odel algebras, we have successfully provided a construction in the style of Fidel-Vakarelov. Furthermore, we establish the existence of a categorical equivalence between the category of monadic G\"odel algebras (or monadic prelinear Heyting algebras) and the category of monadic $N_{c}$-algebras. This construction can be further generalized by removing the prelinearity condition in both classes of algebras. In other words, it can be shown that the category of monadic Heyting algebras satisfying the axiom \ref{G} is equivalent to the category of monadic centered Nelson algebras.

Moreover, the well-known Sendlewski construction (refer to \cite{Sendlewski}) enables us to establish that the category of Nelson algebras is equivalent to the category formed by pairs $({\bf A}, F)$, where ${\bf A}$ is a Heyting algebra and $F$ is a Boolean filter of $A$. An open problem would be extending the Sendlewski construction to the monadic case. To achieve this, an appropriate notion of a monadic Boolean filter should be defined. Categorical equivalences of this type have been established in \cite{F1,F2} for Nelson algebras (or Nelson lattices) equipped with a consistency operator.

\section*{Author Contributions}

All authors contributed to this article.

\section*{Funding}

The work is partially supported by CONICET.

\section*{Data Availability}

This article does not use any particular data, or human participant. Indeed, the results obtained have been established from the articles cited in the references.

\section*{Conflict of interest}

The authors have no conflicts of interest to declare that are relevant to the content of this article.

\section*{Human Participants and/or Animals}

Not applicable.

\section*{Ethical approval}

We declare that we have complied with the ethical standards for publishing articles in this journal.

\begin{acknowledgements}

We extend our gratitude to the editors and reviewers for their valuable contributions to this article. Their insightful feedback and dedicated efforts have greatly improved the manuscript.

\end{acknowledgements}

%
%



\end{document}